\documentclass[12pt]{article}

\usepackage{amsrefs}

\usepackage{amsmath,amsthm,amsfonts,latexsym,amscd,amssymb,enumerate}

\usepackage{hyperref}
\usepackage{color}

\swapnumbers
\theoremstyle{plain}
\newtheorem{theorem}{Theorem}[section]
\newtheorem{lemma}[theorem]{Lemma}
\newtheorem{corollary}[theorem]{Corollary}
\newtheorem{proposition}[theorem]{Proposition}
\newtheorem{conjecture}[theorem]{Conjecture}

\theoremstyle{definition}
\newtheorem{definition}[theorem]{Definition}

\theoremstyle{remark}
\newtheorem{remark}[theorem]{Remark}
\newtheorem{question}[theorem]{Question}

\usepackage{color}

\newcommand{\remts}[1]{#1}

\newcommand{\reals}{\mathbb{R}}
\newcommand{\complexs}{\mathbb{C}}
\newcommand{\naturals}{\mathbb{N}}
\newcommand{\integers}{\mathbb{Z}}
\newcommand{\rationals}{\mathbb{Q}}

\DeclareMathOperator{\Mat}{M}
\DeclareMathOperator{\Zen}{\mathcal{Z}}
\DeclareMathOperator{\SL}{SL}

\newcommand{\abs}[1]{\left\lvert#1\right\rvert} 

\newcommand{\tensor}{\otimes}

\newcommand{\iso}{\cong}

\newcommand{\subgroup}{\leq}
\newcommand{\supergroup}{\geq}

\newcommand{\divides}{\mid}

\DeclareMathOperator{\im}{im}      

\DeclareMathOperator{\tr}{tr}
\DeclareMathOperator{\pr}{pr}

\newcommand{\forget}[1]{}

\newcommand{\innerprod}[1]{\langle #1 \rangle}

{\catcode`@=11\global\let\c@equation=\c@theorem}

\allowdisplaybreaks[2]

\usepackage{amsxtra}

\makeatletter
\let\c@equation=\c@theorem

\makeatother

\DeclareMathOperator{\lcm}{lcm}

\newcommand{\dimu}{\dim^u_G
}
\newcommand{\field}{\mathbb}

\newcommand{\DKG}{D(K[G])}
\newcommand{\Z}{\field{Z}}              

\newcommand{\C}{\field{C}}
\newcommand{\Q}{\field{Q}}

\renewcommand{\NG}{\mathcal{N}(G)}

\newcommand{\NH}{\mathcal{N}(H)}
\newcommand{\UG}{\mathcal{U}(G)}
\newcommand{\EKG}{\mathcal{E}(K[G])}

\newcommand{\ZG}{\mathcal{Z}{(\NG)}}

\newcommand{\tru}{\tr^u_G
}

\newcommand{\CCC}{\mathfrak{C}}

\begin{document}

\pagestyle{myheadings}
\markboth{Anselm Knebusch, Peter Linnell, Thomas Schick}{On the center-valued
  Atiyah conjecture} 


\title{On the center-valued Atiyah conjecture for $L^2$-Betti numbers}

\author{Anselm
  Knebusch\thanks{\protect\href{mailto:anselm.knebusch@hft-stuttgart.de}{\remts{e-mail: anselm.knebusch@hft-stuttgart.de}}}\\HFT-Stutgart\\
University of Applied Sciences\\
Germany \and Peter Linnell\thanks{\protect\href{mailto:plinnell@math.vt.edu}{e-mail: plinnell@math.vt.edu}\protect\\
\protect\href{http://www.math.vt.edu/people/plinnell/}{www:
http://www.math.vt.edu/people/plinnell/}}
\\ Virginia Tech\\ \remts{Blacksburg}\\ \remts{USA} \and Thomas Schick\thanks{
\protect\href{mailto:thomas.schick@math.uni-goettingen.de}{\remts{e-mail:
  thomas.schick@math.uni-goettingen.de}}
\protect\\
\protect\href{http://www.uni-math.gwdg.de/schick}{www:~http://www.uni-math.gwdg.de/schick}\protect\\
{partially funded by the Courant Research Center ``Higher order structures in Mathematics"
within the German initiative of excellence}
}\\
Mathematisches Institut\\
Georg-August-Universit\"at G{\"o}ttingen\\
Germany}
\maketitle

        \begin{abstract}
          The so-called Atiyah conjecture states that the $\NG$-dimensions of
          the $L^2$-homology 
          modules of finite free $G$-CW-complexes belong to a certain set of rational
          numbers, depending on the finite subgroups of $G$.  In this
          article we extend this conjecture to a statement for
          the center-valued dimensions. We show that the conjecture is
          equivalent 
          to a precise description of the structure as a semisimple Artinian
          ring of the division closure
          $D(\mathbb{Q}[G])$ of
          $\Q[G]$  in the ring of affiliated operators. We prove the
          conjecture for all groups in 
          Linnell's class $\CCC$, containing in particular
          free-by-elementary amenable groups. 

          The center-valued Atiyah conjecture states that the center-valued
          $L^2$-Betti numbers of finite free $G$-CW-complexes are contained in
          a certain discrete subset of the center of $\C[G]$, the one
          generated as an additive group by the center-valued traces of all
          projections in $\C[H]$, where $H$ runs through the finite
          subgroups of $G$.

         Finally, we use the approximation theorem of Knebusch
    \cite{Knebusch} for the center-valued $L^2$-Betti numbers to extend
          the 
          result to many groups which are residually in $\CCC$, in particular
          for finite extensions of products of free groups and of pure braid
          groups. 
        \end{abstract}

\section{Introduction}

In \cite{Atiyah}, Atiyah introduced $L^2$-Betti numbers for manifolds with
cocompact free $G$-action for a discrete group $G$ (later generalized to
finite free $G$-CW-complexes). There, he asked \cite{Atiyah}*{p.~72} about the 
possible values these can assume. 
This question was later popularized in precise form as the so-called ``strong
Atiyah conjecture''. One easily sees that the possible values depend on $G$. 
For a finite subgroup of order $n$ in $G$, a free cocompact $G$-manifold
with $L^2$-Betti number $1/n$ can be constructed.
For certain groups $G$ which contain finite subgroups of arbitrarily large
order, with quite some effort manifolds $M$ with $\pi_1(M)=G$ and with
transcendental $L^2$-Betti numbers have been constructed
\cites{Austin,Grabowski,Pichot-Schick-Zuk}. In the following, we will
therefore concentrate on $G$ with a bound on the orders of finite subgroups.

The $L^2$-Betti numbers are defined using the $L^2$-chain complex.
The chain groups there are of the form $l^2(G)^d$, and the differentials are
given by convolution multiplication with a matrix over $\integers[G]$. 
The strong Atiyah conjecture for free finite $G$-CW-complexes
is equivalent to the following (with $K=\integers$): 
\begin{definition}\label{DAtiyah}
	Let $G$ be a group with a bound on the orders of finite
subgroups and let $\lcm(G)\in\naturals$ (the positive integers)
denote the least common multiple of these orders.
	Let $K\subset\complexs$ be a subring.
	
	We say that $G$ satisfies the \emph{strong Atiyah conjecture over $K$,
          or $K[G]$ satisfies the strong Atiyah conjecture} if for every
        $n\in\naturals$ and every $A\in \Mat_n(K[G])$  
	\begin{equation*}
		\dim_G(\ker(A)):=\tr_G(\pr_{\ker A}) \in \frac{1}{\lcm(G)}\integers.  
	\end{equation*}
	Here, as before, we consider $A\colon l^2(G)^n\to l^2(G)^n$ as a
        bounded operator, acting by left convolution multiplication --- the
        continuous extension of the left multiplication action on the group
        ring to $l^2(G)$. $\tr_G$ is the canonical trace on
        $\Mat_n(\NG)$, i.e.~the extension (using the matrix trace) of
        $\tr_G\colon \NG\to\complexs$; $a\mapsto \innerprod{a
          \delta_e,\delta_e}_{l^2(G)}$, where $\NG$, the weak closure of
        $\C[G]\subset \mathcal{B}(\ell^2(G))$ is the group von Neumann algebra.
	
	If $G$ contains arbitrarily large finite subgroups, we set $\lcm(G):=+\infty$.
\end{definition}

A projection $P$ will always be a self adjoint idempotent, so $P =
P^2 = P^*$, where $^*$ indicates the involution on $\mathcal{N}(G)$.
If $E$ is an idempotent, then $E$ is similar to a
projection $P$ and then $\tr_G(E) = \tr_G(P)$.  Also a central
idempotent is always a projection.  Note that if $G$ is
an infinite group, then \remts{the set
$\{\tr_G(P)\}$, where $P$ runs through the projectors in
$\Mat_n(\mathcal{N}(G))$, $n\in\naturals$ consists of all non-negative real numbers}.  The strong Atiyah conjecture
predicts, \remts{on the other hand,} that the $L^2$-Betti numbers take values in
the subgroup of $\reals$ generated by traces of projectors defined already
over $\rationals[H]$ for the finite subgroups $H$ of $G$: the projector
$p_H=(\sum_{h\in H} h)/\abs{H}$ satisfies $\tr_G(p_H)=1/\abs{H}$. And by the
Chinese remainder theorem, the additive subgroup of $\reals$ generated by the
$\abs{H}^{-1}$ is exactly $\frac{1}{\lcm(G)}\integers$. 

We now turn to the center-valued refinements of the above statements. The
center-valued $L^2$-Betti numbers are obtained by replacing the canonical
(com\-plex-valued) trace $\tr_G$ by the center-valued trace $\tru$
(see Definition \ref{def:traces}), taking
values in the center of $\NG$. Note
that by general theory \cite{Kadison2}*{Chapter 8}, as every finite von
Neumann algebra has a unique 
normalized center-valued trace, this is a powerful invariant: two finitely
generated projective Hilbert $\NG$-modules are isomorphic if and only if their
center-valued dimensions coincide.  The center of a ring $R$ will be
denoted $\Zen(R)$.
\begin{definition}\label{def:LG_and_Atiyahconj}
	Let $G$ be a group with $\lcm(G)<\infty$,
let $K$ be a subring of $\mathbb{C}$, let $F$ be the field of
fractions of $K$, and assume that $F$ is closed under complex
conjugation.  Let $\remts{L_K}(G)$ be the additive
        subgroup of $\Zen(\NG)$ generated by $\tru(P)\in
        \Zen(\complexs[G])\subset \Zen(\NG)$ where $P$ runs through projections
        \remts{$P\in F[H]$ with $H\subgroup G$} a finite subgroup.

	We say that $G$ satisfies the \emph{center-valued Atiyah conjecture
          over $K$, 
          or $K[G]$ satisfies the center-valued conjecture} if for every
        $n\in\naturals$ and every $A\in \Mat_n(K[G])$  we have
	$\dimu(\ker(A)):=\tru(\pr_{\ker A})\in \remts{L_K}(G)$.
\end{definition}
Observe that  \remts{$G$
satisfies the center-valued Atiyah conjecture over $K$ if and only if
$G$ satisfies the center-valued conjecture over its field of fractions $F$}.
Indeed the ``only
if'' is obvious.  On the other hand if $A \in \Mat_n(F[G])$, then \remts{(``clearing
denominators'')} there
exists $0 \ne k \in K$ such that $kA \in \Mat_n(K[G])$, and $\ker A =
\ker kA$, which verifies the ``if" part.

\remts{\begin{proposition}
  If a group $G$ satisfies the center-vlaued Atiyah conjecture over $K$ of
  Definition \ref{def:LG_and_Atiyahconj}, then $G$ also satisfies the
  (classical) strong Atiyah conjecture over $K$ of Definition \ref{DAtiyah}.
\end{proposition}
\begin{proof}
  By the universal property of the center-valued trace \cite[Chapter 8]{Kadison2},
  $\tr_G=\tr_G\circ\tru$. We therefore only have to check that $\tr_G(a)\in
  \frac{1}{\lcm(G)}\integers$ for all $a\in L_K(G)$. By the definition of
  $L_K(G)$, we just have to show that $\tr_G(P)\in\frac{1}{\lcm(G)}\integers$
  for each projector $P\in F[H]$, where $H\subgroup G$ is an arbitrary finite
  subgroup. This is of course well known to be true, it follows e.g.~from
  the fact that finite groups satisfy the strong Atiyah conjecture over $K$.
\end{proof}
}

\begin{proposition}[compare Corollary \ref{corol:discrete}]\label{prop:LGdiscrete}
 If $\lcm(G)<\infty$ then $\remts{L_K}(G)\subset \Zen(\NG)$ is discrete. In particular, the
 center-valued Atiyah conjecture predicts a ``quantization'' of the
 center-valued $L^2$-Betti numbers.
\end{proposition}

\begin{remark}
  As for the ordinary strong Atiyah conjecture, the center-valued Atiyah conjecture
  over $\integers [G]$ is equivalent to the statement that the center-valued
  $L^2$-Betti numbers for finite free $G$-CW-complexes take values in $\remts{L_\integers}(G)$.

  The center-valued $L^2$-Betti numbers have been introduced and used in
  \cite{MR1474192}. 
\end{remark}

The strong Atiyah conjecture has many applications. Most interesting are
those for a torsion-free group $G$, i.e.~if $\lcm(G)=1$. This is exemplified
by the following surprising result of Linnell \cite{MR1242889}. We first recall
the notion of ``the'' division closure of $K[G]$.
\begin{definition}
  Let $G$ be a discrete group and let $K\subset \complexs$ be a
subring. Let $\UG$ denote
  the ring of unbounded operators on 
  $l^2(G)$ affiliated to $\NG$ (algebraically, $\UG$ is the Ore localization
  of $\NG$ at the set of all non-zero-divisors). 

  Define the \emph{division
    closure} $D(K[G])$ to be the smallest subring of $\UG$ containing $K[G]$ which
  is closed under taking inverses in $\UG$.
\end{definition}

\begin{theorem}\label{theo:skew_field}
  Let $G$ be a discrete group with $\lcm(G)=1$ and let $K$ be a
subring of $\complexs$.  Then
  $K[G]$ satisfies the strong Atiyah conjecture if and only if $D(K[G])$ is
  a skew field.
\end{theorem}

The appealing feature of this theorem is that it provides a canonical
over-ring, namely $D(K[G])$ of $K[G]$ which should be a skew field, provided $G$ is
torsion free. Observe that this implies in particular that $K[G]$ has no
non-trivial zero-divisors.  For more information on this, see
\cite{Lueck09}*{Remark 4.11}.

Part of the motivation for the work at hand was the question of how to generalize
Theorem \ref{theo:skew_field} if $\lcm(G)>1$. It turns out that one expects
that $D(K[G])$ is semisimple Artinian. In the situation at hand this means that
$D(K[G])$ is a finite direct sum of matrix rings over skew fields. This is
proved in many cases e.g.~in \cite{MR1242889}. 

The present paper \remts{gives} a very precise (conjectural) description of
$D(K[G])$, and if it is satisfied we call $D(K[G])$ \emph{Atiyah-expected
  Artinian}: the 
lattice of finite subgroups and their $K$-linear representations 
give a precise prediction into which matrix summands $D(K[G])$ decomposes
and the size of the corresponding matrices. The precise formula is a bit
cumbersome, so we don't give it here but refer to Definition
\ref{def:Atiyah-Artinean}.

One of our main theorems is the precise generalization of Theorem
\ref{theo:skew_field}.
\begin{theorem}
Let $G$ be a discrete group with $\lcm(G)<\infty$ and let $K$ be a
subfield of $\mathbb{C}$ closed under complex conjugation.
Then $K[G]$ satisfies the center-valued Atiyah conjecture if and only if
  $D(K[G])$ is Atiyah-expected Artinian.
\end{theorem}
Indeed, we show in Theorem \ref{con:at} that these two properties are also
equivalent to the property that $K_0(D(K[G]))$ is generated by the images of
$K_0(K[H])$ as $H$ runs over the finite subgroups of $G$.

\begin{definition}
  Given a discrete group $G$ with $\lcm(G)<\infty$, let $\Delta^+(G)$
denote the
  maximal finite normal subgroup, and let $\Delta(G)$ denote the \emph{finite conjugacy
    center}, i.e.~the set of those elements of $G$ which have only a finite
  number of conjugates.

  Indeed, by \cite{Passman}*{\S 1}, $\Delta(G)$ is a normal subgroup of
  $G$. Recall that the product of two normal subgroups is a normal subgroup,
  therefore, as
  $\lcm(G)<\infty$, $\Delta^+(G)$ makes sense. Note that $\Delta^+(G)\subset
  \Delta(G)$, indeed, using \cite{Passman}*{Lemma 19.3} it is exactly the
  subset of all elements of finite order in $\Delta(G)$.
\end{definition}

In the special case $\Delta^+(G)=\{1\}$, we have that $D(K[G])$ is Atiyah-expected
Artinian if and only 
if it is an $\lcm(G)\times\lcm(G)$-matrix ring over a
skew field, and by Theorem \ref{con:at} this is equivalent to the
center-valued Atiyah conjecture
(which in this case is implied by the usual Atiyah conjecture, as the relevant
part of $\Zen(\NG)$ is $\complexs[\Delta^+(G)]$). This special case (and slightly
more general situations) have already been covered in \cite{L+S2}, but without
the use of the center-valued trace. It turns out that the general case
requires this more refined dimension function. However, much of our arguments
for Theorem \ref{con:at} follow closely the arguments of \cite{L+S2}. 

In \cite{L+S2}, a variant of the division closure, namely the ring $\EKG$ is
introduced and used (compare Definition \ref{def:EKG}). It is closed under
adding central idempotents in $\UG$ 
which generated the same submodules as elements already in the ring. We expect
that this actually coincides with $D(K[G])$. 

\begin{theorem}\label{theo:EKG_equal_DKG}
  If $\lcm(G)<\infty$, $K$ is a subfield of $\complexs$ which is
closed under complex conjugation
  and $G$ satisfies the center-valued Atiyah conjecture, then $\EKG=D(K[G])$.
\end{theorem}

As the second main result of the paper we establish the center-valued Atiyah
conjecture for 
certain classes of groups (namely almost all for which the original Atiyah
conjecture is known).  The algebraic closure of $\mathbb{Q}$ will be
denoted $\overline{\mathbb{Q}}$.

\begin{theorem}\label{theo:good_groups} 
Let $K$ be a subfield of $\mathbb{C}$ which is closed
under complex conjugation.
The center-valued Atiyah conjecture  \remts{over $K$}  is
  true for the following groups $G$:
  \begin{enumerate}
  \item\label{item:linnellsC} all groups $G$ which belong to Linnell's class
    of groups $\CCC$ of Definition \ref{def:CCC}, in 
    particular all free by elementary amenable groups $G$.
  \item\label{item:braids_and_others} if $K$ is contained
in $\overline{\mathbb{Q}}$,  all elementary amenable extensions
  of 
\begin{itemize}
\item pure braid groups
\item  right-angled Artin groups 
\item  primitive link
  groups
\item virtually cocompact special groups, where a ``cocompact special groups''
  is a fundamental group of a compact special cube complex ---this class of
  groups contains Gromov
  hyperbolic groups which act cocompactly and properly on CAT(0) cube complexes,
  fundamental groups of compact hyperbolic $3$-manifolds with empty or
  toroidal boundary, and Coxeter groups
  without a Euclidean triangle Coxeter subgroup,
\item or of products of the above.
\end{itemize}

  \end{enumerate}
\end{theorem}

\begin{question}
  Missing in the above list are congruence subgroups of $\SL_n(\integers)$ and
  finite extensions thereof. Note that the usual Atiyah conjecture for these
  groups, 
  as long as they are torsion free, is proved in \cite{MR2279234}. For
  torsion-free groups, the center-valued Atiyah conjecture is not stronger
  than the usual Atiyah conjecture. However, it would be interesting to
  generalize the work of \cite{MR2279234} to certain extensions which are not
  torsion free, and then (or along the way) to deal with the center-valued
  Atiyah conjecture for these.  
\end{question}

Recall that the center-valued Atiyah conjecture for a group $G$ only makes an
assertion when $\lcm(G)<\infty$. For the proof of \ref{item:linnellsC} of
Theorem \ref{theo:good_groups} we closely follow the method of
\cite{MR1242889}, making use of the equivalent algebraic formulations of the
Atiyah conjecture of Theorem \ref{con:at}. Indeed, we show that the conjecture
is stable under extensions by torsion-free elementary amenable groups. We
actually show (and use) slightly more refined stability properties.

For \ref{item:braids_and_others} of Theorem \ref{theo:good_groups} we use the
approximation theorem for the center-valued $L^2$-Betti numbers,
\cite{Knebusch}*{Theorem 3.2}. Because of the discreteness of the possible
center-valued $L^2$-Betti numbers, the Atiyah conjecture for a suitable
sequence of quotients implies the Atiyah conjecture for the group itself. We
follow here the general idea as already applied in \cite{Schick} and for more
general coefficient rings in \cite{DLMSY}. That this idea can be used for the
class of groups listed in \ref{item:braids_and_others} was shown for the pure
braid groups in \cite{L+S1}, for primitive link groups in \cite{MR2415028} and
for right-angled Coxeter and Artin groups in \cite{LinnellOkunSchick}, and for
cocompact special groups by Schreve in \cite{Schreve} (who uses
fundamentally the geometric insights of Haglund-Wise \cite{Haglund-Wise}, and
develops further the methods of \cite{LinnellOkunSchick}). Agol \cite{Agol}
shows in breakthrough work that 
Gromov hyperbolic cocompact CAT(0) cube  groups are virtually cocompact
special; with Bergeron-Wise' construction of a cocompact action of a
hyperbolic $3$-manifold group on a CAT(0) cube complex \cite{Bergeron-Wise}
this implies that hyperbolic $3$-manifold groups are virtually cocompact
special.

\section{Preliminaries \remts{on rings associated to groups}}\label{sec:prelim}

\subsubsection*{$\UG$, $D(K[G])$ and traces on these}

        \begin{definition}\label{def:traces}
            Let $G$ be a discrete group. The center-valued-trace is
the uniquely
            defined $\mathbb{C}$-linear map 
\[\tru:\NG\rightarrow\Zen(\NG)\] such that for
            $a,b\in \NG$ and $c\in\ZG$, we have
             \begin{itemize}
                    \item $\tru(ab)=\tru(ba)$;
                    \item $\tru(c)=c$;
                    \item $\tru(a)\in (\Zen(\mathcal{N}(G)))^+$ if
$a\in (\mathcal{N}(G))^+$.
             \end{itemize}
        
            The trace can be extended  to $\Mat_d(\NG)$ by taking
            $\tru:=\tru\otimes\tr_{\Mat_d(\C)}$  
            (by abuse of notation), with $\tr_{\Mat_d(\C)}$ the non-normalized
            trace on $\Mat_d(\C)$.

            If $P\in \Mat_d(\NG)$ is a projector with image the (Hilbert
            $\NG$-module) $V$, set $\dimu(V):=\tru(P)$.
         \end{definition}

That a unique such trace exists is established e.g.~in
\cite{Kadison2}*{Chapter 8}.

Later, we want to apply the trace also for the division closure. Recall that
we have (by definition) the following diagram of inclusions of rings
\begin{equation*}
  \begin{CD}
    K[G] @>>> \NG\\
    @VVV @VVV \\
    D(K[G]) @>>> \UG.
  \end{CD}
\end{equation*}

Given a finitely presented $K[G]$-module $M$, represented by $A\in \Mat_{k\times
  l}(K[G])$, i.e.~with exact sequence $K[G]^l\xrightarrow{A} K[G]^k \to M\to 0$, the
induced modules $M\tensor_{K[G]} \NG$, $M\tensor_{K[G]}\UG$, $M\tensor_{K[G]}D(K[G])$
are also finitely presented with the same presenting matrix $A$. The standard
theory of Hilbert $\NG$-modules gives a center-valued dimension for each
finitely presented $\NG$-module, in particular for $M\tensor_{K[G]} \NG$, and
$\dimu(M\tensor_{K[G]} \NG) =k-\dimu(\ker(A))$ in the above situation (compare
\cite{MR1474192}). In \cite{MR1885124}, this dimension is extended to finitely
presented $\UG$-modules, of course in such a way that the value is unchanged
if we induce up from $\NG$ to $\UG$. More precisely, \cite{MR1885124} describes
the extension of dimensions based on arbitrary $\complexs$-valued traces on
$\NG$, this implies easily the corresponding extension for $\dimu$.

\subsubsection*{The central idempotent division closure $\EKG$}

\begin{definition}\label{def:EKG}
  Let $R$ be a subring of the ring $S$ and let $C = \{e \in S \mid e$ is a
  central idempotent of $S$ and $eS = rS$ for some $r \in R\}$.  Then we
  define
  \[
  \mathcal{C}(R,S) = \sum_{e\in C} eR,
  \]
  a subring of $S$.  In the case $S = \mathcal{U}(G)$, we write
  $\mathcal{C}(R)$ for $\mathcal{C}(R,\mathcal{U}(G))$.  For each ordinal
  $\alpha$, define $\mathcal{E}_{\alpha}(R,S)$ as follows:
  \begin{itemize}
  \item $\mathcal{E}_0(R,S) = R$;
  \item $\mathcal{E}_{\alpha+1}(R,S) =
    \mathcal{D}(\mathcal{C}(\mathcal{E}_{\alpha}(R,S),S),S)$;
  \item $\mathcal{E}_{\alpha}(R,S) = \bigcup_{\beta < \alpha}
    \mathcal{E}_{\beta}(R,S)$ if $\alpha$ is a limit ordinal.
  \end{itemize}
  Then $\mathcal{E}(R,S) = \bigcup_{\alpha} \mathcal{E}_{\alpha}(R,S)$.  Also
  in the case $R = K[G]$ where $G$ is a group and $K$ is a subfield of
  $\mathbb{C}$, we  write $\EKG$ for
  $\mathcal{E}(K[G],\mathcal{U}(G))$.
\end{definition}

\begin{conjecture}
  Let $G$ be a discrete group and $K\subset\complexs$ a subfield. Then
  $D(K[G])=\EKG$, at least if $\lcm(G)<\infty$. 
\end{conjecture}

We cite some properties of $\EKG$ from \cite{L+S2} which will be useful
later. Indeed, we generalize from the canonical trace to the center-valued
trace, but the proofs literally also cover this more general situation. 
\begin{lemma}\label{lem:range_of_dim_on_EKG}
(cf.~\cite{L+S2}*{Lemma 2.4})
  The following additive subgroups of $\Zen(\NG)$ coincide:
  \begin{multline*}
    \langle \dimu(x\UG^n)\mid x\in \Mat_n(K[G]), n\in\naturals\rangle\\
    =
    \langle \dimu(x\UG^n)\mid x\in \Mat_n(\EKG), n\in\naturals\rangle
  \end{multline*}
\end{lemma}

This has as an immediate corollary that $\EKG=D(K[G])$ if $K[G]$ satisfies the
center-valued Atiyah conjecture:
\begin{proof}[Proof of Theorem \ref{theo:EKG_equal_DKG}]
  Let $e\in \EKG$ be a central idempotent of $\UG$. Then all the spectral
  projections of $e$ lie in $\Zen(\NG)$, therefore $e$ is affiliated to
  $\Zen(\NG)$. Being an idempotent, even $e\in\Zen(\NG)$. Therefore, on the
  one hand, 
  $\tru(e)=e$ while, on the other hand by Lemma \ref{lem:range_of_dim_on_EKG}, 
  $\tru(e)=\dimu(e\UG)\in \remts{L_K}(G)$, in particular $e\in
  \Zen(K[\Delta^+])\subset K[\Delta^+]$.
\end{proof}

\begin{remark}
  The proof just given didn't need the full force of the center-valued Atiyah
  conjecture, only the statement that $\dimu(x\UG^n)\in
  \Zen(\NG)$ is supported only on elements of finite order, i.e.~lies in
  $\Zen(K[\Delta^+])$. 
\end{remark}

\subsubsection*{Approximation of the center-valued trace}

The following is a special case of
\cite{Knebusch}*{Theorem 3.2}  which will be used in the next section. 

\begin{theorem}\label{theo:approxi}
  Let $G$ be a discrete group with a sequence $G=G_0\supergroup G_1
\supergroup \cdots$ of
  normal subgroups with $\bigcap_{i\in \mathbb{N}} G_i=\{1\}$. 

                       Let $A\in \Mat_d(\overline{\Q}[G])$ and $g\in\Delta(G)$. Let
                       $A[i]\in \Mat_d(\overline \rationals [G/G_i])$ be the image
                       of $A$ under the map induced by the projection
                       $\pr_i\colon G\to G/G_i$.

                       Assume that all $G/G_i$ satisfy the  determinant bound
                       property \cite{DLMSY}*{Definition 3.1}, e.g.~are
                       elementary amenable (or more 
                       generally belong to the class $\mathcal{G}$ of groups
                       introduced in \cite{DLMSY}*{Definition 1.8} and corrected
                       in the errata to \cite{MR1828605} at arXiv:math/9807032, 
                       or are 
                       sofic, compare \cite{Elek+Szabo} and
                       \cite{Knebusch}*{Theorem 4.1}). Then
                       \begin{equation*}
                          \lim_{i\to\infty}\langle
                          \dim^u_{\mathcal{N}(G/G_i)}(\ker(A[i]))
                            ,\pr_i(g) \rangle_{l^2(G/G_i)}=\langle
                            \dimu\ker(A),g\rangle_{l^2(G)}. 
                       \end{equation*}
\end{theorem}

\subsubsection*{Linnell's class $\CCC$}

        \begin{definition}\label{def:CCC} Let $\CCC$ denote the smallest class of groups which
            \begin{enumerate}
                \item contains all free groups,
                \item is closed under directed unions,
                \item satisfies $G\in \CCC$ whenever $H\lhd G$\,, $H\in\CCC$ and $G/H$ is elementary amenable.
            \end{enumerate} 
    
        \end{definition}

\section{Reformulation of the center-valued Atiyah conjecture}

    Let $G$ be a group with $\lcm(G) <\infty$.  We shall assume
that $K$ is
    a subfield of $\mathbb{C}$ which is closed under complex
conjugation.  Many of the arguments given below don't require this
assumption; however if $K$ is a subfield closed under complex
conjugation and $e$ is a central idempotent in $K[G]$, then $e$ is a
projection \cite{BPR10}*{Lemma 9.2(i)}.  Furthermore if $G$ is a
finite group and $A \in \Mat_n(K[G])$, then $\pr_{\ker A} \in
\Mat_n(K[G])$ (use \cite{BPR10}*{Proposition 9.3});
it is here where we are using the property that $K$ is
closed under complex conjugation.

    Recall that $\Delta^+$ is the (finite) normal subgroup consisting
of all elements of finite order and having only finitely many
conjugates.

    \begin{lemma}\label{lemma:xyz} Let $K\subset\complexs$ be
    a subfield which contains all $|\Delta^+|$-th roots of 1, and let
$c_G$ denote the number of finite conjugacy classes of elements of
finite order in $G$,
        i.e.~the dimension of $\Zen(\NG)\cap \Zen(K[\Delta^+])$.
        There is a finite set of primitive central projections
$\{U^1,\dots,
        U^{c_G}\}$ 
        of $\Zen(\NG)\cap \Zen(K[\Delta^+])\subset \Zen(K[G])$, given by
        \begin{equation*}
            U^i:=\sum_{ k \text{ s.t. }\exists g\in G:
                  gu_ig^{-1}=u_k  }u_k,
        \end{equation*}
   where $u_i$ are the primitive central idempotents of the semisimple
   Artinian ring $K[\Delta^+]$.
Furthermore $u_i=\frac{n_i}{|\Delta^+|}\sum_{s\in G}
\chi_i(s^{-1})s$, 
        with $n_i$ the dimensions of the irreducible representations (over
        $\mathbb{C}$) of $\Delta^+$ and $\chi_i$ the corresponding characters
        (extended by $0$ to all of $G$).
Moreover the $U^j$ form
        an orthogonal basis of the vector space $\Zen(\NG)\cap
        \Zen(\complexs[\Delta^+])$.  
\end{lemma}

        \begin{proof}
By Maschke's theorem of standard representation theory, the algebra
$K[\Delta^+]$ is semisimple 
Artinian, compare \cite{Lang1}*{XVIII, Theorem 1.2}. Therefore it has finitely
many primitive central idempotents $u_i$.

Any algebra automorphism must permute the $u_i$, 
in particular the conjugation action of $G$. An element of $\Zen(K[\Delta^+])$
belongs to the 
center of $K[G]$ (and then also of $\NG$) if and only if it is invariant under
conjugation by elements of $G$. It follows
immediately that the $U^i$ are the primitive central idempotents of
$\Zen(\NG)\cap \Zen(K[\Delta^+])$, and furthermore they form an
orthogonal basis for $\Zen(\NG)\cap
        \Zen(\complexs[\Delta^+])$.

The formula for the $u_i$
is also standard, \cite{Lang1}*{XVIII, Proposition 4.4 and Theorem
11.4}.
        \end{proof}

\begin{lemma}\label{lemma:centralidempotents}
Let $K$ be a subfield of $\mathbb{C}$ and let $L/K$ be a finite
Galois extension of $K$ with Galois group $F$.  Let $G$ be a finite
group, let $\{e_1,\dots,e_n\}$ denote the primitive central
idempotents of $K[G]$, and let $\{u_1, \dots, u_m\}$ denote the
primitive central idempotents of $L[G]$.  Then $F$ acts as
automorphisms on $L[G]$ according to the rule $\theta \sum_{g \in G}
a_gg = \sum_{g \in G}\theta(a_g)g$ for $\theta \in F$.
The $u_i$ form an orthogonal set and $\langle
u_i,1\rangle = \langle \theta u_i,1\rangle$ for all $i$.  For each
$i$, define $N_i = \{j \in \mathbb{N} \mid e_iu_j = u_j\}$.  Then $F$
acts transitively on $\{u_j \mid j \in N_i\}$ and $e_i =
\sum_{j \in N_i} u_j$.
\end{lemma}
\begin{proof}
This is well-known, and follows from Galois descent.
Note that $u_ie_j$ is a central idempotent in $L[G]$ and $u_i = u_ie_j
+ (1-e_j)u_i$.  It follows for all $i,j$, either $u_ie_j = 0$ or
$u_ie_j = u_i$, because $u_i$ is primitive.  It follows easily that
$e_i = \sum_{j\in N_i} u_j$.  Also $F$ acts on $\{u_j \mid j\in
N_i\}$, and the sum of the $u_j$ in an orbit is fixed by $F$ and is
therefore in $K[G]$.  Since $e_i$ is primitive, it follows that this
orbit must be the whole of $N_i$.  Finally if $e = \sum_{g\in G}
e_gg \in L[G]$ is an idempotent, then $e_1 \in \mathbb{Q}$ (by the character
formula of Lemma \ref{lemma:xyz}) and we see
that $\langle u_i,1\rangle = \langle \theta u_i,1\rangle$ for all $i$.
\end{proof}

    \begin{lemma}\label{lemma:xyz1} Let $K\subset\complexs$ be a subfield,
let $\omega$ be a primitive $|\Delta^+|$-root of 1 and
set $L = K(\omega)$.  Let $F$ denote the Galois group of $L$ over
$K$, and let $U^1,\dots, U^{c_{L[G]}}$ be the primitive central
projections of $\Zen(\NG)\cap \Zen(L[\Delta^+])\subset \Zen(L[G])$ as
described above in Lemma~\ref{lemma:xyz}.  There is a finite set of
primitive central projections
$\{P^1,\dots,
        P^{C_{K[G]}}\}$
        of $\Zen(\NG)\cap \Zen(K[\Delta^+])\subset \Zen(K[G])$, given by
        \begin{equation*}
            P^i:=\sum_{ k \text{ s.t. }\exists g\in G:
gp_ig^{-1}=p_k  }p_k,
        \end{equation*}
   where $p_i$ are the primitive central idempotents of the
semisimple
   Artinian ring $K[\Delta^+]$.  Set $N_i = \{j \in \mathbb{N} \mid P^i
U^j = U^j\}$.  Then
\[
P^i = \sum_{j \in N_i}  U^j
\]
and $F$ acts transitively on $\{U^j \mid j \in N_i\}$.
\end{lemma}
\begin{proof}
This follows from Lemmas \ref{lemma:xyz} and
\ref{lemma:centralidempotents}.
\end{proof}

    \begin{lemma}\label{lemma:Atiyah}
    Let $H$ be a finite subgroup of $G$ which contains $\Delta^+$.
        For an irreducible projection $Q\in K[H]$ (in the sense that if
        $Q=Q_1+Q_2$ with projections in $Q_1,Q_2\in K[H]$ satisfying
        $Q_1Q_2=0$ then either $Q_1=0$ or $Q_2=0$)  we have $\tru(Q)\in
        \Zen(\NG)\cap \Zen(K[\Delta^+])\subset \Zen(\NG)$. More precisely, using
        the central projections $P^i$ of Lemma \ref{lemma:xyz1} we have
        \begin{equation}\label{eq:Atiyah}
            \tru(Q)  =
                     \frac{\dim_\complexs(Q\cdot
                       \complexs [H])\cdot |\Delta^+|}{|H|\cdot
                       \dim_\complexs(P^i\cdot
                       \complexs[\Delta^+])}P^i   
                     =\frac{\dim_{\NG}(Q\cdot l^2(G))}{\dim_{\NG}(P^i\cdot l^2(G))}P^i
        \end{equation}
        where $P^i$ is characterized by the
        property $QP^i=Q$.
    \end{lemma}
        \begin{proof}
Let $\omega$ be a primitive $|\Delta^+|$-th root of 1, let $L =
K(\omega)$ and let $F$ denote the Galois group of $L/K$.
          The center-valued trace is obtained by orthogonal projection from
          $l^2(G)$ to the subspace of $l^2(\Delta)$ spanned by functions which
          are constant on $G$-conjugacy classes, using the standard embedding of
          $\NG$ 
          into $l^2(G)$. For $Q$, which is supported on group
          elements of finite order, therefore $\tru(Q)\in \complexs[\Delta^+]$.
Let $U^1,\dots, U^{c_G}$ and $P^1,\dots,P^{C_{K[G]}}$ be the primitive
central projections as described in Lemma \ref{lemma:xyz1}.
          Using the standard inner product on $\complexs [H]$ we obtain,
          using that $(U^1,\dots, U^{c_G})$ is an orthogonal basis of
          $\Zen(\NG)\cap 
          \Zen(L[H]) = \Zen(\NG)\cap \Zen(L [\Delta^+])$ 
          \begin{equation}\label{eq:trGQ}
            \tru(Q)=\sum_j \frac{\innerprod{Q,U^j}}{\innerprod{U^j,U^j}} U^j.
          \end{equation}
            Moreover, we have for each $j$ that $QP^j+Q(1-P^j)=Q$ and
            $QP^jQ(1-P^j)=0$, the latter because $P^j$ is central. Since 
            $Q$ is irreducible, 
            we get either $QP^j=Q$ or $QP^j=0$.
If $QP^i=Q$ we have $Q\sum_{j \in N_i}U^j = Q$ and $QU^j = 0$ for $j
\notin N_i$.  Also if $j_1,j_2 \in N_i$, then $\theta (QU^{j_1}) =
QU^{j_2}$ for some $\theta \in F$ and we see that $\langle QU^{j_1},
1\rangle =\langle QU^{j_2},1\rangle$, consequently $\langle
Q,U^j\rangle$ is independent of $j$ for $j\in N_i$.  Similarly
$\langle U^j,U^j\rangle$ is independent of $j$ for $j \in N_i$.
Thus $\langle Q,P^i\rangle = |N_i| \langle Q,U^j\rangle$,
$\langle P^i,P^i \rangle = |N_i| \langle U^j,U^j\rangle$ for $j \in
N_i$, hence
\[
\frac{\innerprod{Q,U^j}}{\innerprod{U^j,U^j}} =
\frac{\innerprod{Q,P^i}}{\innerprod{P^i,P^i}}.
\]
Substitute this in equation \eqref{eq:trGQ} together with
            \begin{align*}
                \langle Q,P^j\rangle    & =\langle QP^j,1\rangle =\langle
                Q,1\rangle  =\frac{\dim_\C(Q \cdot \mathbb{C}[H])}{|H|}\\
                \langle P^j,P^j\rangle  &
               = \innerprod{P^j,1}=\frac{\dim_\C(P^j \cdot \mathbb{C}[\Delta^+])}{|\Delta^+|}. 
              \end{align*}
              These formulas  follow from the character formula for
              projections or are directly obtained as follows:  for a
              projection $P\in\complexs [E]$ 
		   and a finite group $E$ we have $\innerprod{P,1}_{l^2(E)} 
              =\innerprod{Ph,h}_{l^2(E)}$ for all $h\in E$, therefore
              $\dim_{\complexs}(P\cdot \complexs [E])=\tr(P)=\sum_{h\in E}
              \innerprod{Ph,h}=\abs{E}\cdot \innerprod{P,1}$.

              Note, finally, that $\frac{\dim_\complexs(Q\cdot \complexs
                [H])}{\abs{H}} = \dim_{\NH}(Q\cdot l^2(H)) = \dim_{\NG}(Q\cdot
              l^2(G))$ by the induction rule for von Neumann dimensions.
        \end{proof}

    \begin{corollary}\label{corol:discrete}
      The additive subgroup $\remts{L_K}(G)$ of $\Zen(\NG)$ of Definition
      \ref{def:LG_and_Atiyahconj} is discrete.
    \end{corollary}
    \begin{proof}
      \remts{Recall that $F$ denotes the relevant subfield of $\complexs$ in the
      setup of Definition \ref{def:LG_and_Atiyahconj}}, namely $F$ is
the field of fractions of $K$. Given a finite subgroup $H$ of $G$ and a \remts{projection $P\in F[H]$},
      $\tru(P)$ is a positive integral linear combination of $\tru(Q_\alpha)$
      where $Q_\alpha \in \remts{F[H]}$ are irreducible projections, corresponding to
      the decomposition of $\im(P)$ into irreducible $\remts{F[H]}$-modules.

      It therefore suffices to check that the additive subgroup of
$\Zen(\NG)$
      generated by $\tru(Q)$ is discrete, where $Q$ runs through the
      irreducible projections in $\remts{F[H]}$ and $H$ runs through the finite
      subgroups of $G$. Increasing the field and increasing the finite
      subgroup has the only potential effect that an irreducible projection
      breaks up as a sum of new irreducible projections and therefore the
      subgroup generated by their center-valued
      traces increases. Therefore we may assume that these subgroups
contain $\Delta^+$ and that $\remts{F = \mathbb{C}}$.
By Lemma \ref{lemma:Atiyah}, these are all integer
      multiples of $\lcm(G)^{-1}P^i$ with the orthogonal basis
$(P^1,\dots,
      P^{c_G})$, therefore span a discrete subgroup of $\Zen(\NG)$.
    \end{proof}

    \begin{definition}\label{def:Atiyah-Artinean}
      Assume that $G$ is a discrete group with $\lcm(G)<\infty$ and that
      $K$ is a subfield of $\mathbb{C}$ which is closed under complex
conjugation.

      We say that $D(K[G])$ is \emph{Atiyah-expected Artinian} if it is a
      semisimple Artinian ring such that its primitive central idempotents are
      the central idempotents $P^1,\dots,P^{C_{K[G]}} \in
K[\Zen(K[\Delta^+])]$ of Lemma
\ref{lemma:xyz1},
      and if each direct summand $P^jD(K[G])P^j$ is an $L_j\times L_j$ matrix
      ring over 
      a skew field.

      Here, $L_j$ is determined as follows: consider all irreducible
      sub-projections $Q_\alpha\in K[H_\alpha]$ of 
      $P^j$ (i.e.~those satisfying $Q_\alpha P^j=Q_\alpha$), where $H_\alpha$ runs
      through all finite subgroups of $G$ containing $\Delta^+$.
By Lemma \ref{lemma:Atiyah},
      $\tru(Q_\alpha)=q_\alpha P^j$ for some rational number $q_\alpha$. Because
      there are only 
      finitely many isomorphism classes of finite subgroups of $G$, 
      formula \eqref{eq:Atiyah} shows that the collection of these rational
      numbers is finite. $L_j$ is the smallest integer such that each
      $q_\alpha$ is 
      an integer multiple of $\frac{1}{L_j}$. Explicitly,
            \begin{equation*}
                L_j=\frac{\dim_\C(P^j\cdot
                  \complexs[\Delta^+])\lcm(G)}{
\gcd\left( \dim_\C(P^j\cdot
                  \complexs[\Delta^+])\lcm(G),
                  \dim_\complexs(Q_\alpha\cdot\complexs
                  [H_\alpha])\frac{\lcm(G)}{|H_\alpha|} \abs{\Delta^+}\mid
                \alpha\right) } \in\Z.
            \end{equation*}
\end{definition}
\begin{proof}
  We have to show that the two descriptions of $L_j$ coincide, using
  Equation \eqref{eq:Atiyah}, i.e.~we have to find the smallest common
  denominator of all these fractions. We expand the denominators to the common
  value $\lcm(G)\cdot \dim_\complexs(P^j\cdot \complexs[\Delta^+])$, then we
  have to divide this by the greatest common divisor of this number and of
  all the new numerators.
\end{proof}

\begin{theorem}\label{con:at}
        Let $G$ be a discrete group, with $\lcm(G) < \infty$ and let $K\subset
        \C$ be a subfield closed under complex conjugation.
        The following statements are equivalent.
        \begin{enumerate}
            \item\label{item:At-art} $\DKG$ is Atiyah-expected Artinian as in
              Definition \ref{def:Atiyah-Artinean}.
\item\label{item:KKG} $\phi\colon \bigoplus_{E\le G\,:\,|E|<\infty} K_0(K[E])
  \rightarrow K_0(\DKG)$ is surjective and $D(K[G])$ is semisimple Artinian. 
 \item\label{item:GKG} $\phi\colon\bigoplus_{E\le G\,;\,|E|<\infty}
              G_0(K[E])\rightarrow G_0(D(K[G]))$ is surjective.
 \item\label{item:quant}
$KG$ satisfies the center-valued Atiyah conjecture.
            \end{enumerate}
            Recall here that, for a ring $R$, $K_0(R)$ is the Grothendieck group
            of finitely generated projective $R$-modules, whereas $G_0(R)$ is
            the Grothendieck group of arbitrary finitely generated
            $R$-modules. 
\end{theorem}

\begin{proof}[Proof of Theorem \ref{con:at}]
        \ref{item:At-art} $\implies$ \ref{item:KKG}: We use the notation of
        Definition \ref{def:Atiyah-Artinean}. Using the row projectors of
        matrix rings, there are projections
$x_1,\dots,x_{C_{K[G]}}\in
        D(K[G])$ which represent a $\integers$-basis of the free abelian group
        $K_0(D(K[G]))$, and $[P^i]=L_i[x_i]$ in $K_0(D(K[G]))$. We have to show
        that each $x_i$ is an integer linear combination of images of elements
        of $K_0(K[H_\alpha])$ with $H_\alpha$ finite. If $Q_\alpha\in K[H_\alpha]$
        is an irreducible sub-projection of $P^i$, then $\phi([Q_\alpha])$ is a
        multiple of $[x_i]$ in $K_0(D(K[G]))$, namely (comparing the
        center-valued dimensions which are defined for finitely generated
        projective $D(K[G])$-modules by the discussion of Section
        \ref{sec:prelim}) $\phi([Q_\alpha])=q_\alpha [P^i]$ if 
        $\tru(Q_\alpha)=q_\alpha P^i$. By the Chinese remainder theorem and
        the definition of $L_i$ as the smallest integers such that all the
        $q_\alpha$ are integer multiples of $L_i^{-1}$, also
        $[x_i]=L_i^{-1}[P^i]$ belongs to the image of $\phi$.



\ref{item:KKG} $\implies$ \ref{item:GKG}: For a semisimple Artinian ring
        every finitely generated module 
        is projective, therefore $G_0 = K_0$ under the assumptions we make.

        \ref{item:GKG} $\implies$ \ref{item:quant}: Let $M$ be a finitely
        presented $K[G]$-module with presentation $K[G]^l\xrightarrow{A} K[G]^n\to
        M\to 0$, $A\in M_{n\times l}(K[G])$. Then  $M\tensor_{K[G]}
        D(K[G])$ is finitely generated, therefore by the assumption stably
        isomorphic to an 
        integer linear combination $\bigoplus a_i x_iD(K[G])$ with $x_i$
        projectors defined over finite subgroups $E$ of $G$ --- note that
        $G_0(K[E])=K_0(K[E])$ for any finite group $E$, as $K[E]$ is semisimple
        Artinian. Inducing further 
        to $\UG$ and using that the dimension function extends to finitely
        presented $\UG$-modules (which is additive, so that we can leave out
        the stabilization summands), 
        we read off that
        \begin{equation*}
\dimu(M)=\dimu(\bigoplus a_i x_i\UG)
 =\sum a_i
        \dimu(x_i\UG) \in \remts{L_K}(G)
      \end{equation*}
by definition of $\remts{L_K}(G)$. Finally, by additivity of the von Neumann dimension  $\dimu(\ker(A)) =
n-\dimu(M)\in \remts{L_K}(G)$.

\ref{item:quant} $\implies$ \ref{item:At-art}: Here, we follow closely the argument of the proof of
\cite{L+S2}*{Proposition 2.14}.
   Our assumption implies by Theorem \ref{theo:EKG_equal_DKG} that
   $\EKG=D(K[G])$. Because the center-valued Atiyah conjecture implies that
   the ordinary $L^2$-Betti numbers are contained in a finitely generated
   subgroup of $\rationals$ (generated by $\tr_G(P^j)/L_j$),  by
   \cite{L+S2}*{Theorem 2.7} $D(K[G])$ is a semisimple Artinian ring. 

The $P^j$ are central idempotents in $D(K[G])$. We have
        to show that they are primitive central idempotents, and that each is
        the sum of exactly $L_j$ orthogonal sub-idempotents which are
        themselves irreducible. The structure theory of rings then implies
        that each $P^jD(K[G])P^j$ is simple Artinian and an $L_j\times
        L_j$-matrix ring over a skew field. 

 Fix, as in Definition \ref{def:Atiyah-Artinean}, the (finite) collection of
 sub-projections $Q_\alpha$ of $P^j$, 
        where the $Q_\alpha$ are irreducible projections supported on
        $K[H_\alpha]$ and $H_\alpha$ runs through the (isomorphism classes of)
        finite extensions of $\Delta^+(G)$ inside $G$.  Then
$\tru(Q_\alpha)=\frac{n_\alpha}{L_j} P^j$ with integers $n_\alpha$, and by
definition of $L_j$ we have $\gcd_\alpha(n_\alpha)=1$. Set
$d:=\lcm_\alpha(n_\alpha)$.

Consider now $P^j\UG^d$. Because
$$\dimu(P^j\UG^d)=dP^j=\dimu(Q_\alpha\UG^{L_jd/n_\alpha})$$ by 
\cite{Lueck02}*{Theorem 9.13(1)} then $P^j\UG^d\iso
Q_\alpha\UG^{L_jd/n_\alpha}$, so we find $L_jd/n_\alpha$ mutually orthogonal
projections in $\Mat_d(\UG)$ corresponding to the copies of $Q_\alpha$. Because
the center-valued trace of each of those equals
$\frac{n_\alpha}{L_j}P^j=\tru(Q_\alpha)$, by \cite{Sterling}*{Exercise
  13.15A}, there exist $L_jd/n_\alpha$
similarities (i.e.~self-adjoint unitaries)
$u_i \in \mathcal{U}(G)$ with $u_1 = 1$ such that these projections can be
written as $u_iP'_\alpha u_i$ (where $P_\alpha'$ is the diagonal matrix with
first entry $P_\alpha$ and all other entries $0$). 

Then, exactly as in the proof of \cite{L+S2}*{Proposition 2.14} we can replace
the $u_i$ by $\tilde u_{i}\in \Mat_d(D(K[G]))$ which are invertible and such that we
still have a direct sum decomposition
\begin{equation}\label{eq:dirsum}
  P^jD(K[G])^d = \bigoplus_{i=1}^{L_jd/n_\alpha} \tilde u_{i} P_\alpha' D(K[G])^d.
\end{equation}
This uses
the Kaplansky density theorem, the quantization of the center-valued trace and
\cite{L+S2}*{Lemma 2.12}.

Let us now take a central idempotent $\epsilon$ in $D(K[G])$ which is a sub-projection of
$P^j$ (i.e.~$\epsilon P^j=\epsilon$). We have to show that $\epsilon=0$ or
$\epsilon=P^j$. To do this, we 
compute $\tru(\epsilon)$. Note that all the modules $\epsilon \tilde
u_iP_\alpha'\UG^d$ are isomorphic, therefore by Equation \eqref{eq:dirsum}
\begin{equation}\label{eq:tru_eps}
  d\tru(\epsilon)=\dimu(\epsilon\UG^d) = \frac{L_jd}{n_\alpha} \dimu(\epsilon
  P_\alpha' \UG^d).
\end{equation}
By Lemma \ref{lem:range_of_dim_on_EKG} and the assumption \ref{item:quant},
$L_j\cdot \dimu(\epsilon P_\alpha'\UG^d)$ is an integer multiple of
$P^j$. Therefore, rearranging Equation \eqref{eq:tru_eps}
\begin{equation*}
  n_\alpha\tru(\epsilon) \in \integers P^j.
\end{equation*}
As this holds for all $\alpha$, even
\begin{equation*}
 \epsilon= \tru(\epsilon) =\lcm_\alpha(n_\alpha)\tru(\epsilon)\in \integers P^j.
\end{equation*}
So we can indeed conclude that $P^j$ is a primitive central
idempotent and
therefore $P^jD(K[G])$ is an $l\times l$ matrix ring over a skew field. It
follows that $P^jD(K[G])^{n_\alpha d}$ is the direct sum of $n_\alpha d l$
copies of an irreducible submodule. On the other hand, $P^j
D(K[G])^{n_\alpha d}$ is the direct sum of $L_jd$ isomorphic summands
for every $\alpha$. As
$\lcm_\alpha(n_\alpha)=1$ we conclude that $L_j\divides l$. On the other hand,
by the assumption \ref{item:quant} and Lemma \ref{lem:range_of_dim_on_EKG},
the center-valued dimension of the irreducible submodule (which is generated
by one projector as $P^jD(K[G])$ is Artinian) is an integer multiple of
$L_j^{-1}P^j$ and therefore $L_j\divides l$. It follows that $l=L_j$ as
claimed.
\end{proof}

\section{Special cases and inheritance properties of the center-valued Atiyah conjecture}

Throughout this section, we assume that $K$ is a subfield of
$\mathbb{C}$ which is closed under complex conjugation.



\begin{lemma}\label{lem:finite}
   The center-valued Atiyah conjecture is true for finitely generated
virtually free groups.
\end{lemma}
\begin{proof}
This follows from the proof of \cite{MR1242889}*{Proposition 5.1(i)
and Lemma 5.2(ii)} (in
which $\mathbb{C}$ can be replaced by any subfield of $\mathbb{C}$)
and Theorem \ref{con:at}\ref{item:KKG}.
\end{proof}

\begin{lemma}\label{prop:union}
  If $G$ is a directed union of groups $G_i$ and the center-valued Atiyah
  conjecture over $K$ is true for all groups $G_i$, then it is also true for
  $G$. 
\end{lemma}
\begin{proof}
  By \cite{MR1242889}*{Lemma 5.3}, $D(K[G])$ is the directed union of the
  $D(K[G_i])$. Any matrix $A$ over $D(K[G])$ is therefore already a matrix over
  $D(K[G_i])$ for some $i$, with $\dim^u_{G_i}(\ker(A))\in \remts{L_K}(G_i)$. Composition
  with the center-valued trace for $G$ gives (by the induction formula for von
  Neumann dimensions) $\dimu(\ker(A)) \in \tru(\remts{L_K}(G_i))\subset \remts{L_K}(G)$.
\end{proof}


\begin{proposition}\label{prop:extension}
  Assume that we have an extension $1\to H\to G\xrightarrow{\pi} E\to 1$ where
  $E$ is 
  elementary amenable and for each finite subgroup $F\subgroup E$,
  $\pi^{-1}(F)\subgroup G$ satisfies the center-valued Atiyah conjecture over
  $K$ . Then
  also $K[G]$ satisfies the center-valued Atiyah conjecture.
\end{proposition}
\begin{proof}
  By transfinite induction, the statement is a formal consequence of the same
  assertion where $E$ is finitely generated virtually abelian, as explained
  e.g.~in the proof of \cite{Schick}*{Proposition 3.1} or in \cite{MR1242889}.

  If $E$ is finitely generated virtually abelian then in the proof of
  \cite{MR1242889}*{Lemma 5.3} it is shown that 
            $$\bigoplus_{F\subgroup E \text{ finite}} G_0(D(K[\pi^{-1}(F)]))\to
            G_0(D(K[G]))$$ 
            is onto, using  Moody's induction
            theorem \cite{Moody}*{Theorem 1}. Since by assumption
            $\bigoplus_{U\subgroup 
              \pi^{-1}(F)\text{ finite}}G_0(K[U])\to G_0(D(K[\pi^{-1}(F)]))$ is
            onto for each such $F$ and the composition of surjective maps is
            surjective we conclude that
            $$\bigoplus_{F\in \mathcal{F}(G)} G_0(K[F])\to G_0(D(K[G]))$$
            is onto and \ref{item:GKG} of Theorem \ref{con:at} is established.  
\end{proof}

\begin{proposition}\label{prop:inverse_limit}
Let $K$ be a subfield of $\overline{\mathbb{Q}}$ which is closed
under complex conjugation.
  Assume that $G$ is a group with a sequence $G \supergroup
G_1\supergroup \cdots$ of
  normal subgroups such that $\bigcap_{i\in\naturals} G_i=\{1\}$. Assume
  moreover that for each $i\in \naturals$ and each finite subgroup $F\subgroup
  G/G_i$ there is a finite subgroup $F'\subgroup G$ which is mapped
  isomorphically to $F$ by the projection $G\to G/G_i$.

  Finally, assume that each $G/G_i$ satisfies the determinant bound conjecture
  and the center-valued Atiyah
  conjecture over $K$. Then $K[G]$ satisfies the center-valued Atiyah
  conjecture. 
\end{proposition}
\begin{proof}
  As the statement is empty if $\lcm(G)=\infty$, we assume that
  $\lcm(G)<\infty$. 
  We first show that, if $i$ is large enough,
  $\pi_i$ induces an isomorphism $\pi_i\colon \Delta^+(G)\to
  \Delta^+(G/G_i)$. Dropping finitely many terms in the sequence we can then
  assume that this is the case for all $i\in\naturals$. To prove the
  assertion, choose a finite 
  subgroup $M$ of $G$ with maximal order (possible since
  $\lcm(G)<\infty$). Note that the product $\Delta^+M$ is also a finite
  subgroup, 
  therefore by maximality equal to $M$, consequently $\Delta^+\subgroup M$. Then
  choose finitely many
  $g_1,\dots,g_n\in G$ such that $\Delta^+(G)=\bigcap_{k=1}^n
M^{g_k}$ (where $M^g$ denotes the conjugate $gMg^{-1}$),
which is possible by the descending chain condition for finite sets.

  Finally, choose $r>0$ such that $\pi_r\colon G\to G/G_r$ is injective when
  restricted to $\bigcup_{k=1}^n M^{g_k}$, which is possible because
  $\bigcap_{i} G_i=\{1\}$.

  Because $\pi_r$ is surjective, $\pi_r(\Delta^+(G))$ is a finite normal
  subgroup 
  of $G/G_i$ and therefore $\pi_r(\Delta^+(G))\subgroup \Delta^+(G/G_r)$. On the
  other hand, $\pi_r(M)$ is a finite subgroup with maximal order in $G/G_r$
  (because $\pi_r|_M$ is injective and every finite subgroup of $G/G_r$ is an
  isomorphic image of a finite subgroup of $G$), therefore
  $\Delta^+(G/G_r)\subgroup \pi_r(M)$, by normality even
  $\Delta^+(G/G_r)\subgroup \bigcap_{k=1}^n \pi_r(M)^{\pi_r(g)}$. As
  $\bigcap_{k=1}^n M^g =\Delta^+(G)$ and by injectivity of $\pi_r$ on
  $\bigcup_{k=1}^n M^g$ we finally get
  \begin{equation*}
    \Delta^+(G/G_r)\subgroup \bigcap_{k=1}^n \pi_r(M)^{\pi_r(g)} =
    \pi_r(\Delta^+(G)) \subgroup \Delta^+(G/G_r).
  \end{equation*}
This implies the statement for all $i\ge r$.

Secondly, given $g\in G$ of infinite order, for all sufficiently large $i$,
the restriction of $\pi_i$ to $\{1,g,g^2,\dots, g^{\lcm(G)}\}$ is injective
and therefore, as by assumption the orders of finite subgroups of $G/G_i$ are
bounded by $\lcm(G)$, $\pi_i(g)$ also has infinite order.

Fix now $A\in \Mat_d(K[G])$ and denote by $Q_i$ the projection onto the kernel
 of $A[i]:=p_i(A)$. Recall that
\begin{equation*}
\tru(Q_i)=\dimu(\ker(A))=\sum_{g\in
  G}\innerprod{\dimu(\ker(A)),g}_{l^2(G)}g,
\end{equation*}
and we denote by
$\innerprod{\dimu(\ker(A)),g}$ the \emph{coefficient of $g$} in
$\dimu(\ker(A))$, and correspondingly for $\tru(Q_i)$. 

The center-valued Atiyah conjecture for $K[G/G_i]$ implies in particular that
$\tru(Q_i)$ is contained in $K[\Delta^+(G/G_i)]$, therefore supported only on
elements of finite order. Consequently, if $g\in G$ has infinite order, then
$\innerprod{\tru(Q_i),\pr_i(g)}=0$ for sufficiently large $i$ and, by Theorem
\ref{theo:approxi}, $\innerprod{\dimu(\ker(A)),g}=0$. This implies that $\dimu(\ker(A))$ is
supported on elements of finite order, i.e.~is contained in $\Zen(\NG)\cap
K[\Delta^+(G)]$. 

As explained above, we can use $\pi_i$ to identify $\Delta^+(G)$ and
$\Delta^+(G/G_i)$ and consider $\tru(Q_i)$ as an element of $K[\Delta^+(G)]$. By
Theorem \ref{theo:approxi}, for each $g\in \Delta^+(G)$,
\begin{equation*}
  \innerprod{\dimu(\ker(A)),g} =\lim_{i\to\infty} \innerprod{\tru(Q_i),g}.
\end{equation*}
Since all the (finitely many) coefficients converge, we even have
\begin{equation*}
\lim_{i\to\infty} \tru(Q_i) = \dimu(\ker(A))\in \Zen(\NG)\cap
K[\Delta^+(G)].
\end{equation*}
Because the sets of isomorphism classes of finite
subgroups 
of $G/G_i$ and of $G$ are identified by $\pi_i$, we get exactly the
same relevant irreducible projections defined over finite subgroups and the
same central idempotents in the formulas of Lemma \ref{lemma:xyz} and Lemma
\ref{lemma:Atiyah} for $\remts{L_K}(G)$ and $\remts{L_K}(G/G_i)$. Consequently, $\pi_i$ identifies
$\remts{L_K}(G)$ and 
$\remts{L_K}(G/G_i)$. Finally, observe that by assumption about the Atiyah
conjecture for $G/G_i$ we have  $\tru(Q_i)\in \remts{L_K}(G)$. As the
latter is a discrete subset of $\Zen(\NG)$, we finally observe that
$\dimu(\ker(A))\in \remts{L_K}(G)$, i.e.~$K[G]$ satisfies the center-valued Atiyah
  conjecture. 
\end{proof}

\begin{theorem}\label{theo:LiC}
        The center-valued Atiyah conjecture is true for all groups $G\in\CCC$.
      \end{theorem}
\begin{proof}
        In the proof of
            \cite{MR1242889}*{Lemma 4.9} it is shown that the assertion
            follows (by transfinite induction) directly from Lemma
            \ref{lem:finite}, 
            Lemma \ref{prop:union} and Proposition \ref{prop:extension}.
\end{proof}

\begin{corollary}\label{corollary}
Let $K$ be a subfield of $\overline{\mathbb{Q}}$ which is closed
under complex conjugation.
  Then the center-valued Atiyah conjecture is true for all elementary amenable extensions of pure braid groups, of 
  right-angled Artin groups, of primitive link groups, of cocompact special
  groups, or of products of such.
\end{corollary}
\begin{proof}
  Each of the groups in the list has a sequence of normal subgroups with
  trivial intersection and with elementary amenable quotients such that in
  addition the condition of Proposition \ref{prop:inverse_limit} is met. This
  is shown for the extensions of pure braid groups in \cite{L+S1}, for
  primitive link groups in \cite{MR2415028} and for right-angled Coxeter and
  Artin groups in \cite{LinnellOkunSchick}, and combining \cite{Schreve} with
  \cite{LinnellOkunSchick} it also follows for special cocompact
  groups. Combining Theorem \ref{theo:LiC} 
 and Proposition \ref{prop:inverse_limit}, the assertion
  follows. 
\end{proof}


\end{document}